\documentclass[12pt]{amsart}
\usepackage{amssymb,latexsym,amsmath,amsthm,enumitem,mathrsfs,geometry}
\usepackage{fullpage}
\usepackage{fancyhdr}
\usepackage{color}
\usepackage{stmaryrd}
\usepackage{mathrsfs}
\usepackage{hyperref}
\usepackage{lineno}
\usepackage{diagbox}
\usepackage{array}
\usepackage{tikz}
\usetikzlibrary{arrows}
\usetikzlibrary{positioning}
\usepackage{float}
\usepackage{pgfplots}
\usepackage{comment}

\newtheorem{theorem}{Theorem}[section]
\newtheorem*{theorem*}{Theorem}
\newtheorem{lemma}{Lemma}[section]
\newtheorem*{remark*}{Remark}
\newtheorem{theoremletter}{Theorem}
 \newtheorem{corollary}{Corollary}[section]
\newtheorem{proposition}[theorem]{Proposition}

\definecolor{pink}{rgb}{1,.2,.6}
\definecolor{orange}{rgb}{0.7,0.3,0}
\definecolor{blue}{rgb}{.2,.6,.75}
\definecolor{green}{rgb}{.4,.7,.4}
\definecolor{purple}{RGB}{127,0,255}

\newcommand{\ayla}[1]{{\color{purple}{#1}} }

\begin{document}

\numberwithin{equation}{section}

\title{Partitions into prime powers}

\author[Gafni]{A. Gafni}
\address{Department of Mathematics, University of Mississippi, 305 Hume Hall, University MS 38677}
\email{argafni@olemiss.edu}

\keywords{partitions, restricted partition functions, prime powers, exponential sums, Hardy-Littlewood circle method}
\subjclass[2010]{ 11P55, 11P82, 11L07, 11L20}
\thanks{}

\date{\today}

\begin{abstract} For a subset $\mathcal A\subset \mathbb N$, let $p_{\mathcal A}(n )$ denote the restricted partition function which counts partitions of $n$ with all parts lying in $\mathcal A$.  In this paper, we use a variation of the Hardy-Littlewood circle method to provide an asymptotic formula for $p_{\mathcal A}(n )$, where $\mathcal A$ is the set of $k$-th powers of primes (for fixed $k$).  This combines Vaughan's work on partitions into primes with the author's previous result about partitions into $k$-th powers.  This new asymptotic formula is an extension of a pattern indicated by several results about restricted partition functions over the past few years.  Comparing these results side-by-side, we discuss a general strategy by which one could analyze $p_{\mathcal A}(n )$ for a given set $\mathcal A$.

\end{abstract}

\maketitle 

\section{Introduction and background} \label{intro}

A \emph{partition} of a number  $n$ is a non-increasing sequence of positive integers whose sum is equal to $n$.  The number of partitions of $n$ is denoted by the \emph{partition function} $p(n)$.  The asymptotic study of partitions began in 1918 with the seminal result of Hardy and Ramanujan, showing that
\begin{equation}\label{Hardy Ramanujan}
p(n) \sim \frac{1}{4n\sqrt{3}} \exp\left(\pi\sqrt{\frac23}n^{\frac12}\right) ,
\end{equation}
as $n\rightarrow \infty.$  For a subset $\mathcal A\subset \mathbb N$, we let $p_{\mathcal A}(n )$ denote the \emph{restricted partition function} which counts partitions of $n$ with all parts lying in $\mathcal A$.  In this paper, we study partitions into $k$-th powers of primes.  The main result (Theorem \ref{main result full asymp}) is an asymptotic formula for the number of such partitions.  Because of the sparsity and irregularity of powers of primes, the formula is given in terms of quite complicated auxiliary functions.  The result can be simplified to the following asymptotic equivalence:
\begin{theorem}\label{main result pattern form}
Fix $k\in \mathbb N$ and let $\mathbb{P}_k = \{ p^k :  p \text{ prime}\}.$  There exist positive constants $C_1, C_2$, depending only on $k$, such that the number of partitions of $n$ with all parts lying in $\mathbb{P}_k $ satisfies
\begin{equation*}
p_{\mathbb{P}_k}(n) \sim  C_1 n^{-\frac{2k+1}{2k+2}} (\log n)^{-\frac{k}{2k+2}} \exp\left(C_2\frac{n^{\frac{1}{(k+1)}}}{ (\log n)^{\frac{k}{(k+1)}}}(1+o(1))\right),
\end{equation*}
as $n\rightarrow\infty$.
\end{theorem}
\begin{remark*} This may seem to be an unusual statement, since it is possible that the negative powers of $n$ and $\log n$ may be dwarfed by the error term in the exponential.  We state the result this way in order to illustrate the similarity in form between this asymptotic and the results below.  The error term in the exponential is a result of the fact that the prime number theorem cannot be expressed in terms of elementary functions.  
\end{remark*}

\subsection{A pattern of results}
Theorem \ref{main result pattern form} is an extension of a pattern indicated by several results over the past few years.  
In a previous paper, the author proved the following asymptotic for partitions into $k$-th powers\footnote{This result was originally proven by Wright \cite{Wright1934} using substantially more complicated techniques.  Vaughan \cite{Vaughan2015} used a simpler method, outlined in Section \ref{strategy section}, to give a new proof in the case $k=2$.  The author generalized Vaughan's method to give a new proof of Wright's result for general $k$.}: 
\begin{theoremletter}[\cite{Gafni2016}]
Let  $\mathcal A_k = \{x^k : x \in \mathbb N\}  $.  Then 
$$p_{\mathcal A_k} (n) \sim C_1 \exp\left( C_2 n^{\frac{1}{k+1}})\right)n^{-\frac{3k+1}{2(k+1)}},$$
where $C_1, C_2$ are positive constants depending only on $k$.
\end{theoremletter}
\noindent
Berndt, Malik, and Zaharescu generalized that result to partitions into $k$-th powers in a residue class:
\begin{theoremletter}[Berndt-Malik-Zaharescu, \cite{BeMaZa2018}]
Let $\mathcal A_{k, (a,b)} = \{x^k : x\equiv a\pmod b, x \in \mathbb N\}  $.  Then 
$$p_{\mathcal A_{k, (a,b)}} (n) \sim C_1 \exp\left( C_2 n^{\frac{1}{k+1}}\right)n^{-\frac{b+bk +2ak}{2b(k+1)}}$$
where $C_1, C_2$ are positive constants depending only on $k$, $a$, and $b$.
\end{theoremletter}
\noindent
The sets $\mathcal A_k$ and $\mathcal A_{k, (a,b)}$ can be thought of as integer values of the polynomials $x^k$ and $(bx+a)^k$, respectively.  Dunn and Robles further extended this idea to study partitions into integer values of a polynomial: 
\begin{theoremletter}[Dunn-Robles, \cite{DunnRob2018}]
Let $f$ be a polynomial and let  $\mathcal A_f = \{f(x) : x \in \mathbb N\}$.  If $\mathcal A_f \subset \mathbb N$  and $\gcd(\mathcal A_f) = 1$, then 
$$p_{\mathcal A_f} (n) \sim C_1 \exp\left( C_2 n^{\frac{1}{d+1}}\right)n^{-\frac{2d(1-\zeta(0, \alpha))+1}{2(d+1)}}$$
where $d = \deg(f)$, $\zeta(0, \alpha)$ is a value of an appropriate Matsumoto-Weng $\zeta$ function, and $C_1, C_2$ are positive constants depending only on the polynomial $f$.  
\end{theoremletter}
We can also consider $p_{\mathcal A}(n)$ for sets that are not induced by polynomials.  In 2008, Vaughan proved the following result about partitions into primes:
\begin{theoremletter}[Vaughan, \cite{Vaughan2008}]
Let $\mathcal A = \mathbb P$ be the set of primes.  Then
$$p_{\mathbb P} (n) \sim C_1 \exp\left( C_2 \frac{n^{\frac12}}{(\log n)^{\frac12}}(1+o(1))\right)n^{-\frac{3}{4}}(\log n)^{-\frac14}$$
where $C_1, C_2$ are positive constants.
\end{theoremletter}
\noindent

\begin{remark*} Different constants $C_1, C_2$ are used in each of the results above.  The original papers state more explicit versions of the asymptotic formulae, and it is possible to compute the constants from those theorems.  We omit the explicit expressions for $C_1, C_2$ here because they are quite complicated and do not provide significant insight to this discussion.
\end{remark*}

Comparing the above results we see that the dominant term of $\log p_\mathcal{A}(n)$ is given only in terms of the growth of $\mathcal A$.  Indeed, for partitions into polynomial values we have
$$\log p_{\mathcal{A}_f}(n) \sim C_2 n^{\frac{1}{d+1}} = C_2 (n^{\frac{1}{d}})^{\frac{d}{d+1}},$$
and for partitions into powers of primes we have
$$\log p_{\mathbb{P}_k}(n) \sim C_2 \frac{n^{\frac{1}{d+1}} }{(\log n)^{\frac{d}{d+1}}} = C \left(\frac{n^{\frac1d}}{\frac{1}{d}\log n}\right)^{\frac{d}{d+1}}.$$
All of these results are proved using the same variation of the Hardy-Littlewood circle method.  In Section \ref{strategy section}, we outline the common technique and discuss what information is needed to prove an analogous result for a given set $\mathcal A$.  

It should be noted that the results above are not the first formulas for restricted partition functions.  Throughout the 20th Century, restricted partitions were studied extensively by a number of mathematicians, including Wright \cite{Wright1934}, Roth and Szekeres \cite{RotSze1954}, and Bateman and Erd\H{o}s \cite{BatErd1956Mono}.  The generality and strength of these results vary, as do the methods employed.  The results stated above are highlighted because they are proved using a parallel framework and they exhibit the possibility of a pattern for the analysis of other restricted partition functions.   

In recent work, Debruyne and Tenenbaum \cite{DebTen2020} use the saddle-point method to study the asymptotics of  restricted partitions in a general setting.  The saddle-point method can be viewed as a crude version of the circle method, in which the only major arc is the principal major arc at the origin.  Debruyne and Tenenbaum give an asymptotic estimate of $p_{\mathcal A}(n)$ for $\mathcal A$ in a certain class of sets, which includes the sets $\mathcal A_k$, $\mathcal A_{k,(a,b)}$, and $\mathcal A_f$.   The set $\mathbb{P}_k$ of $k$-th powers of primes is not included in their class of sets.  Indeed, one of their requirements is that the Dirichlet series $\sum_{a\in\mathcal A} a^{-s}$ can be meromorphically continued to the half-plane $\Re(s) \ge -\varepsilon$.  In the case of $\mathcal A  = \mathbb{P}_k$, that Dirichlet series is $\mathcal P(ks) = \sum_p p^{-ks}$, which has logarithmic singularities at every zero of the Riemann $\zeta$-function.  The general method for studying restricted partitions, outlined in Section \ref{strategy section} below, is less rigid than the method employed in  \cite{DebTen2020} and is successful for a broader class of sets.

\subsection{The full asymptotic formula}\label{setup subsec}
The generating function for partitions into $k$-th powers of primes is 
\begin{equation*}
\Psi(z) = \sum_{n\ge 0}p_{\mathbb P_k}(n) z^n = \prod_{p \text{ prime}} \left(1-z^{p^k}\right)^{-1}.
\end{equation*}
It will be more convenient to write this as
\begin{equation*}
\Psi(z) = \exp(\Phi(z)).
\end{equation*}
where
\begin{equation*}
\Phi(z) = \sum_{j=1}^\infty \sum_{p \text{ prime}} \frac{1}{j} z^{jp^k}.
\end{equation*}
By Cauchy's integral formula, we have
\begin{equation}\label{Cauchy integral}
p_{\mathbb P_k}(n)  = \rho^{-n} \int_{-1/2}^{1/2} \Psi(\rho e(\alpha)) e(-n\alpha) \, d\alpha = \rho^{-n} \int_{-1/2}^{1/2} \exp(\Phi(\rho e(\alpha)) )e(-n\alpha) \, d\alpha,
\end{equation}
for any positive real number $\rho<1$. Let $x\in \mathbb R$ be large.  (Eventually we will set $x=n$.)  We choose $\rho = \rho(x)$ so that 
\begin{equation*}x = \rho \Phi'(\rho).\end{equation*} 
It will follow from Lemma \ref{rho d/drho lemma} that the relationship between $x$ and $\rho$ is well-defined and injective, and that $\rho\rightarrow 1^-$ as $x\rightarrow\infty$.

In order to interpret the result of Theorem \ref{main result full asymp} and see that it implies Theorem \ref{main result pattern form}, we will need estimates for the auxiliary functions involved.  

\begin{proposition}\label{growth prop}
As $x\rightarrow \infty$, we have
\begin{equation}\label{x log(1/rho)}
x\log\frac{1}{\rho(x)} = \left(\frac{\frac{k+1}{k}\zeta(\frac{k+1}{k})\Gamma(\frac{k+1}{k}) x^{\frac1k}}{\log x}\right)^{\frac{k}{k+1}}\left(1-\frac{k}{k+1}\frac{\log\log x}{\log x} + O\left(\frac{1}{\log x}\right) \right),
\end{equation}
\begin{equation}\label{phi rho}
\Phi(\rho(x)) = k \left(\frac{\frac{k+1}{k}\zeta(\frac{k+1}{k})\Gamma(\frac{k+1}{k}) x^{\frac1k}}{\log x}\right)^{\frac{k}{k+1}} \left(1-\frac{k}{k+1}\frac{\log\log x}{\log x} + O\left(\frac{1}{\log x}\right) \right),
\end{equation}
and, for $m\ge 1$,
%\begin{equation}
%\Phi^{(m)}(\rho(x)) = x^{\frac{mk+1}{k+1}}\left(\frac{\log x}{\zeta\left(1+\frac1k\right)\Gamma\left(2+\frac1k\right)}\right)^{\frac{k(m-1)}{k+1}}\frac{\Gamma\left(m+\frac1k\right)}{\Gamma(1 + \frac1k)} \left(1+O\left(\frac{\log\log x}{\log x}\right)\right),
%\end{equation}
%and 
\begin{equation}\label{Phi_m estimate}
\Phi_{(m)}(\rho(x)) = x^{\frac{mk+1}{k+1}}\left(\frac{\log x}{\frac{k+1}{k}\zeta(\frac{k+1}{k})\Gamma(\frac{k+1}{k})}\right)^{\frac{k(m-1)}{k+1}}\frac{\Gamma\left(m+\frac1k\right)}{\Gamma(1 + \frac1k)} \left(1+O\left(\frac{\log\log x}{\log x}\right)\right),
\end{equation}
where 
\begin{equation*} 
\Phi_{(m)}(\rho) = \left(\rho \frac{d}{d\rho}\right)^m \Phi(\rho).
\end{equation*}
\end{proposition}

We are now ready to state the asymptotic formula.

\begin{theorem}\label{main result full asymp} Using the notation defined above with $\rho = \rho(n)$, we have
\begin{equation*}
p_{\mathbb P_k} (n) = \frac{\rho^{-n} \Psi(\rho)}{\sqrt{2\pi \Phi_{(2)}(\rho)}} \left(1+ O(n^{-\frac{1}{2k+3}})\right).
\end{equation*}
%where 
%\begin{equation*} 
%\Phi_{(2)}(\rho) = \left(\rho \frac{d}{d\rho}\right)^2 \Phi(\rho).
%\end{equation*}
\end{theorem}

\subsection{Proof of Theorem \ref{main result pattern form} given Theorem \ref{main result full asymp}}
By Proposition \ref{growth prop}, we have 
\begin{align*} 
\rho^{-n} \Psi(\rho)&  = \exp\left(n\log\frac{1}{\rho(n)} + \Phi(\rho(n))\right) \\ &
 = \exp \left((k+1)\left(\zeta\left(1+\frac1k\right)\Gamma\left(2+\frac1k\right)\right)^{\frac{k}{k+1}} \frac{n^{\frac{1}{(k+1)}}}{ (\log n)^{\frac{k}{(k+1)}}}(1+o(1)) \right), 
\end{align*}
and 
\begin{equation*}
\sqrt{\Phi_{(2)}(\rho(n))} = n^{\frac{2k+1}{2k+2}}\left(\frac{\log n}{\zeta\left(1+\frac1k\right)\Gamma\left(2+\frac1k\right)}\right)^{\frac{k}{2k+2}}\left(1 + \frac1k\right)^{\frac12} (1+o(1)).
\end{equation*}
Therefore, Theorem \ref{main result full asymp} implies that
\begin{align*}
p_{\mathbb P_k} (n)& = \frac{\rho^{-n} \Psi(\rho)}{\sqrt{2\pi \Phi_{(2)}(\rho)}} \left(1+ O(n^{-\frac{1}{2k+3}})\right) \\&
\sim  C_1 n^{-\frac{2k+1}{2k+2}} (\log n)^{-\frac{k}{2k+2}} \exp\left(C_2\frac{n^{\frac{1}{(k+1)}}}{ (\log n)^{\frac{k}{(k+1)}}}(1+o(1))\right)
\end{align*}
where 
\begin{align*} C_1 & = \sqrt{2\pi}\left(1 + \frac1k\right)^{\frac12}\left({\zeta\left(1+\frac1k\right)\Gamma\left(2+\frac1k\right)}\right)^{\frac{-k}{2k+2}}, \\ 
C_2 &= (k+1)\left(\zeta\left(1+\frac1k\right)\Gamma\left(2+\frac1k\right)\right)^{\frac{k}{k+1}} .
\end{align*}

\subsection{The difference function}
 The methods used to prove the asymptotic formula in Theorem \ref{main result full asymp} can also be used to estimate the growth of $p_{\mathbb{P}_k}(n)$.  This yields the following:

\begin{theorem}\label{difference theorem} Using the notation defined above with $\rho = \rho(n)$, we have
\begin{equation*}
p_{\mathbb{P}_k}(n+1) - p_{\mathbb{P}_k}(n) \sim \frac{\rho^{-n} \log(\frac{1}{\rho})\Psi(\rho)}{\sqrt{2\pi \Phi_{(2)}(\rho)}} \left(1+ O(n^{-\frac{1}{2k+3}})\right).
\end{equation*}
\end{theorem}
\noindent
From Theorem \ref{difference theorem} and Proposition \ref{growth prop} we  immediately deduce an asymptotic equivalence:

\begin{corollary} \label{difference cor}  Let $\rho(n)$ be as defined above.  Then
\begin{equation*}\label{difference equiv}
p_{\mathbb{P}_k}(n+1) - p_{\mathbb{P}_k}(n) \sim    \left(\left(1+\frac1k\right)\zeta\left(1+\frac1k\right)\Gamma\left(1+\frac1k\right)\right)^{\frac{k}{k+1}} \frac{ p_{\mathbb P_k} (n) }{(n\log n)^{\frac{k}{k+1}}},
\end{equation*}
as $n\rightarrow\infty$.
\end{corollary}

\subsection{Organization and notation.}
In Section \ref{strategy section}, we give an outline of a method to study $p_\mathcal{A}(n)$ for a given set $\mathcal A$.  Section \ref{rho analysis section} provides an analysis of the function $\rho(x)$ and a proof of Proposition \ref{growth prop}.  The main result (Theorem \ref{main result full asymp}) is proved in Sections \ref{setup sec}, \ref{major arcs sec}, and \ref{end of proof sec}.  We prove Theorem \ref{difference theorem} in Section \ref{difference sec}.  Section \ref{appendix} contains some additional analysis of the exponential sum $S_k^*(q,a)$, included for the sake of completeness.

We use the the Vinogradov notation $f\ll g$ to mean that there exists a positive constant $C$ such that $|f|\leq C|g|$.  We write $f=g+O(h)$ to denote that $|f-g|\ll h$.  The asymptotic equivalence $f\sim g$ means that $\lim f/g = 1$.  We also use the standard notation $e(\alpha) = e^{2\pi i \alpha}$.  When working with $p_{\mathbb P_k} (n)$, the power $k$ is assumed to be fixed in advance and any implicit constants may depend on $k$.

\subsection*{Acknowledgements}  The author is grateful to Amita Malik for many fruitful conversations related to this paper, and to the anonymous referee for helpful suggestions.

%%%%%%%%%%---------------------------------
\section{A general method for studying restricted partitions}\label{strategy section}
The results listed in Section \ref{intro} exhibit a pattern for the growth of restricted partition functions in general.  In fact, the proofs of these results are parallel in many respects and give way to a method by which one could analyze $p_{\mathcal A}(n)$ for a range of different sets $\mathcal A$.  In this section we summarize that method and identify the information about $\mathcal A$ required to implement it.

The technique is based on the Hardy-Littlewood circle method (See \cite{Vaughan1997}).  
\begin{comment}
The generating function for $p_{\mathcal A}(n)$ is 
 $$\Psi_{\mathcal{A}}(z) := \sum_{n=0}^\infty p_{\mathcal{A}}(n)z^n =  \prod_{a\in\mathcal{A}} (1 - z^{a})^{-1}.$$
It is more convenient to work with an infinite sum than an infinite product, so we write
$$\Psi_{\mathcal{A}}(\rho e^{2\pi i \theta}) = \exp(\Phi_{\mathcal{A}}(e^{-1/X} e(\theta))),$$

\end{comment}
Adjusting the setup in Section \ref{setup subsec} to a general set $\mathcal A$, we see that
\begin{equation}\label{circle method integral}
 p_{\mathcal{A}}(n) = \int_0^1 \rho^{-n} \exp\left(\Phi_{\mathcal{A}}(\rho e(\theta))\right) e(\theta) \,d\theta, 
 \end{equation}
where $\rho <1 $ and 
$$\Phi_{\mathcal{A}}(z) = \sum_{j=1}^\infty \sum_{a\in\mathcal A } \frac{ z^{aj}}{j} z^{aj}.$$

 In a typical implementation of the circle method, one would split the intergral \eqref{circle method integral} into the major arcs and the minor arcs, and all of the major arcs would contribute to the main term of the asymptotic formula. However, in the case of restricted partitions functions, the contribution from the major arc at the origin is significantly greater than the contribution from the rest of the major arcs. So, we split the integral into three main parts, namely
 $$p_{\mathcal{A}}(n) = \left\{\int_{\mathfrak{M}(1,0)} + \int_{\mathfrak{M}\setminus\mathfrak{M}(1,0)} + \int_{\mathfrak{m}} \right\}  \rho^{-n} \exp\left(\Phi_{\mathcal{A}}(\rho e(\theta))\right) e(\theta)  \,d\theta.$$
 We treat the remaining major arcs  in the traditional way, but the major arcs ${\mathfrak{M}(q,a)}$ with $q>1$ do not contribute to the main term of the asymptotic formula.  
 %Rather they are  ``thrown away'' into the error term.  
 The main term comes exclusively from the first part of the integral, when $\theta$ is close to the origin.  
 
 The analysis of \eqref{circle method integral} requires different information about the set $\mathcal A$ in each of the three settings.  To evaluate the behavior on the principal major arc $\mathfrak{M}(1,0)$, we need analytic information about the \emph{Dirichlet series}  
$$\sum_{a\in\mathcal A} a^{-s},$$ 
including convergence properties, analytic continuation, zeros, singularities, and residues at poles.  To analyze the non-principal major arcs, we need to understand the \emph{distribution of $\mathcal A$ in residue classes}.
 Finally, to estimate the contribution from the minor arcs, we need bounds for certain \emph{Weyl sums over  $\mathcal A$}: 
$$\sum_{a\le y, \, a\in \mathcal A}e(ja\theta).$$

If we understand these three ingredients for a given set $\mathcal A$, we can use the techniques of this paper to analyze $p_{\mathcal A}(n)$. The difficulty is that this information is quite elusive and it is rare to find a set $\mathcal A$ for which we have a good understanding of all three pieces.   As we will see in the proof of Theorem 1.3, the current state of knowledge about the distribution of prime powers is only just sufficient to make this process work for $\mathcal A = \mathbb P_k$.

%%%%%%%%%%---------------------------------
\section{Proof of Proposition 1.3}\label{rho analysis section}
In this section we obtain estimates on $\Phi(\rho)$ that will be necessary for the main result.   For convenience, we define a parameter $X=\frac{1}{\log \frac{1}{\rho}}$, so that $\rho = e^{-1/X}$.  
\begin{lemma}\label{rho d/drho lemma}
Let $\rho = e^{-1/X}$.  Then, for $m\in\mathbb{Z}_{\ge0}$, we have
\begin{equation}\label{rho d/drho}
\left(\rho \frac{d}{d\rho}\right)^m \Phi(\rho)  = \frac{\zeta\left(1+\frac1k\right) \Gamma\left(m+\frac1k\right) X^{m+\frac1k}}{\log X}\left(1+ O\left(\frac{1}{\log X}\right)\right)
\end{equation}
and 
\begin{equation}\label{Phi derivatives}
\Phi^{(m)}(\rho)  = \frac{\zeta\left(1+\frac1k\right) \Gamma\left(m+\frac1k\right) X^{m+\frac1k}}{\log X}\left(1+ O\left(\frac{1}{\log X}\right)\right),
\end{equation}
as $\rho\rightarrow 1^-$.
\end{lemma}

\begin{proof}
We have
\begin{equation*}
\left(\rho \frac{d}{d\rho}\right)^m \Phi(\rho) =  \sum_{j=1}^\infty \sum_p j^{m-1} p^{km} \rho^{jp^k} = \sum_{j=1}^\infty \sum_p j^{m-1} p^{km} e^{-jp^k/X}.
\end{equation*}
Using a Mellin transform and interchanging the sums with the integral, we obtain
\begin{equation}\label{mellin}
\left(\rho \frac{d}{d\rho}\right)^m \Phi(\rho) =  \frac{1}{2\pi i} \int_{c-i\infty}^{c+i\infty} X^s \mathcal P(k(s-m)) \zeta (s+1-m) \Gamma(s) \, ds,
\end{equation}
where $\mathcal P(s) = \sum_p p^{-s} $ and $c>m+1$.  Note that $\mathcal P(s) = \log \zeta(s) - \mathcal D(s)$, where 
$$\mathcal D(s) = \sum_{j=2}^\infty \frac1j \sum_p \frac{1}{p^{js}}.$$
For any $\delta>0$, $\mathcal D(s)$ converges absolutely and uniformly for $\mathfrak{R}(s) \ge \frac12 +\delta$.  If we replace $\mathcal P(k(s-m))$ by $\mathcal D(k(s-m))$ in \eqref{mellin}, we can move the line of integration to the line $\mathfrak{R}(s) = c_0$ for any $c_0 > m +\frac{1}{2k}$.  We can then crudely bound the contribution from this piece of the integral by
\begin{equation}\label{D contribution}
\frac{1}{2\pi i} \int_{c_0-i\infty}^{c_0+i\infty} X^s \mathcal D(k(s-m)) \zeta (s+1-m) \Gamma(s) \, ds \ll X^{c_0}.
\end{equation}

We now consider the rest of the integral in \eqref{mellin}, namely
\begin{equation*}
\frac{1}{2\pi i} \int_{c-i\infty}^{c+i\infty} X^s \log\left( \zeta(k(s-m))\right) \zeta (s+1-m) \Gamma(s) \, ds.
\end{equation*}
The integrand here is analytic in the zero-free region for $\zeta(k(s-m))$ except for a logarithmic singularity at $s=m + \frac{1}{k}$.  Choosing $T = \exp(\sqrt{\log X})$, the integral can be truncated at height $T$ with an acceptable error.   The remaining part of the integral can be moved left to the line $\mathfrak{R}(s) = m + \frac1k - \frac{c}{\log T}$ (where $c$ is a suitably small positive constant), except for a keyhole loop around the singularity at $s = m+\frac1k$.  The loop runs along the bottom and top edges of the branch cut over $\{ s = \sigma : \sigma \le m+\frac1k\}$.   Aside from the parts along the cut, all pieces of this contour are well-controlled and will contribute only to the error term.  Along the bottom and top edges of the cut, the value of $\log\zeta(s)$ differs by $2\pi i$ while the value of  $X^s \zeta(s+1 - m) \Gamma(s)$ remains the same.  Thus the contribution from the pieces along the cut is 
\begin{equation*}
\frac{1}{2\pi i } \int_0^{\frac{c}{\log T}} 2\pi i \, X^{m + \frac1k -u}\, \zeta\left(1+ \frac1k - u\right) \Gamma\left(m+ \frac1k -u\right) \, du.
\end{equation*}
For any $m$, we have
$$ \zeta\left(1+ \frac1k - u\right)\Gamma\left(m + \frac1k  -u\right)  = \zeta\left(1+ \frac1k \right)\Gamma\left(m +\frac1k\right) + O(u)$$
uniformly for $u\in[0,1/2]$, so the contribution from the cuts is
\begin{equation*}
\frac{X^{m+ \frac1k}}{\log X}\zeta\left(1+ \frac1k\right)\Gamma\left(m+\frac1k\right) \left(1 +O\left( \frac{1}{\log X}\right)\right).
\end{equation*}
Choosing $c_0 = m + \frac{3}{4k}$, say, in \eqref{D contribution}, gives \eqref{rho d/drho}.  To obtain \eqref{Phi derivatives}, we first notice that the case $m=0$ is immediate. By induction on $m$, we have 
\begin{equation*}
\rho^{m} \Phi^{(m)}(\rho)  = \sum_{i=1}^m c_{i,m} \left(\rho \frac{d}{d\rho}\right)^i \Phi(\rho),
\end{equation*}
where the $c_{i,m}$ are real numbers with $c_{m,m} =1.$  Since $\rho = 1 + O(1/X)$, the result follows.
\end{proof}

\subsection{Proof of Proposition \ref{growth prop}}
Suppose that $x$ is sufficiently large.  Then $\rho$ is close to $1$ and so $X= \frac{1}{\log(\frac1\rho)}$ is also large.
By Lemma \ref{rho d/drho lemma}, we see that 
\begin{equation}\label{x in terms of X}
x = \rho \frac{d}{d\rho} \Phi(\rho)  = \frac{\zeta\left(1+\frac1k\right) \Gamma\left(1+\frac1k\right) X^{\frac{k+1}{k}}}{\log X}\left(1+ O\left(\frac{1}{\log X}\right)\right)
\end{equation}
Thus 
\begin{equation*}
\log x =  \log(X^{\frac{k+1}{k} }) -\log\log X + O(1)
\end{equation*}
Since $\log x \ll \log X \ll \log x$, we have that $\log\log X = \log\log x +O(1)$. Thus
\begin{equation*}
\log X  = \frac{k}{k+1} \log x + \frac{k}{k+1} \log\log x + O(1)  
 %\\ &   = \frac{k}{k+1}\log x \left(1 + \frac{\log\log x}{\log x} +O ( \frac{1}{\log x})\right).
\end{equation*}
%Going back to \eqref{x in terms of X}, we find
%\begin{align}
%X^{\frac{k+1}{k}} & = \frac{x\log X}{\zeta\left(1+\frac1k\right) \Gamma\left(1+\frac1k\right) }\left(1+ O\left(\frac{1}{\log X}\right)\right) \\ &
%	=\frac{k}{k+1} \frac{x\log x}{\zeta\left(1+\frac1k\right) \Gamma\left(1+\frac1k\right) } \left(1 + \frac{\log\log x}{\log x} +O ( \frac{1}{\log x})\right), 
%\end{align}
Putting this into \eqref{x in terms of X} and solving for $X$, we find
\begin{equation}\label{X in terms of x}
X = \left(\frac{k}{k+1} \frac{x\log x}{\zeta\left(1+\frac1k\right) \Gamma\left(1+\frac1k\right) } \right)^{\frac{k}{k+1}} \left(1 + \frac{k}{k+1}\frac{\log\log x}{\log x} +O \left( \frac{1}{\log x}\right)\right).
\end{equation}
The fact that $x\log \frac{1}{\rho} =xX^{-1}$ yields  \eqref{x log(1/rho)}.  Combining \eqref{x in terms of X} with Lemma  \ref{rho d/drho lemma}, we obtain 
\begin{equation}\label{transition xX}
\Phi_{(m)} (\rho) = \left(\rho \frac{d}{d\rho}\right)^m \Phi(\rho) = xX^{m-1} \frac{\Gamma\left(m+\frac1k\right)}{\Gamma\left(1+\frac1k\right)} \left(1+ O\left(\frac{1}{\log x}\right)\right).
\end{equation}
Plugging \eqref{X in terms of x} into \eqref{transition xX} gives \eqref{phi rho} when $m=0$ and \eqref{Phi_m estimate} when $m\ge 1$.
%\begin{equation}
%\rho (x)  = 1 + O(\frac{1}{X}) = 1 + O((x\log x)^{-\frac{k}{k+1}}),
%\end{equation}
%and
%\begin{equation}
%\Phi'(\rho) = \frac{x}{\rho}  =  x + O(x^{\frac{1}{k+1}}(\log x)^{-\frac{k}{k+1}}).
%\end{equation}
\newline \null \hfill$\square$

%%%%%%%%%%----------------------------------------%%%%%%%%%%%%%%%%%%%%
\section{Proof of Theorem \ref{main result full asymp}:  The Setup}\label{setup sec}
%Recall that $p_{\mathbb P_k}(n)$ is equal to the integral in \eqref{Cauchy integral}.
We define the major and minor arcs as follows.  Set
\begin{equation*}
\delta_q = q^{-1} X^{-1}(\log X)^{12k}, \quad Q = (\log X)^{12k},  
\end{equation*}
and for $1\le a\le q \le Q$ with $(a,q)=1$, define
\begin{equation*}
\mathfrak{M}(q,a) = \left(\frac{a}{q}  - \delta_q,   \frac{a}{q}  + \delta_q\right).
\end{equation*}
The major arcs  $\mathfrak M$ and the minor arcs $\mathfrak m$ are given by
$$\mathfrak M= \bigcup_{\substack{1\le a\le q\le Q\\ (a,q)= 1}} \mathfrak{M}(q,a), \quad \mathfrak m = \left[-1/2,1/2\right)\setminus\mathfrak{M}.$$

 As discussed in Section \ref{strategy section}, we divide the integral \eqref{Cauchy integral} into three pieces: the principal major arc $\mathfrak{M}(1,0)$, the non-principal major arcs $\mathfrak{M}(q,a)$ with $q>1$, and the minor arcs $\mathfrak m$.  The main term will come exclusively from principal major arc, which is analyzed in Section \ref{end of proof sec}.  Our bound for the denominators in the major arcs is limited by the scope of the Siegel-Walfisz theorem.  We choose the exponent of $\log X$ to be $12k$ to obtain a satisfactory bound on the minor arcs.

We begin with some preliminary analysis of $\Phi(\rho e(\alpha))$ that will be used throughout the proof of Theorem \ref{main result full asymp}.  From the definition, we have
\begin{equation*}
\Phi( \rho e(\alpha)) = \sum_{j=1}^\infty \frac1j  \sum_p e^{-p^k j / X} e(j p^k \alpha). 
\end{equation*}
We write 
\begin{equation*}
 e^{-p^k j / X}= \int_p^\infty k y^{k-1} j X^{-1} e^{-y^k j / X} \, dy.
\end{equation*}
Thus 
\begin{equation*}
\sum_p e^{-p^k j / X} e(j p^k \alpha) =  \int_2^\infty k y^{k-1} j X^{-1} e^{-y^k j / X} \sum_{p\le y} e(j p^k \alpha)\, dy.
\end{equation*}
It is useful to observe the crude bound 
\begin{equation*}
 \int_2^\infty k y^{k-1} j X^{-1} e^{-y^k j / X} \sum_{p\le y} e(j p^k \alpha)\, dy  \ll  \int_0^\infty y\, ky^{k-1} j X^{-1} e^{-y^k j / X}\, dy.
\end{equation*}
Using integration by parts, we have that for any  $\lambda > 0$,
\begin{equation} \label{integral for any lambda}
\int_2^\infty y^\lambda \left(k y^{k-1} j X^{-1} e^{-y^k j / X} \right)\, dy \ll  \left(\frac{X}{j}\right)^{\lambda/k}.
\end{equation}
Let $J$ be a parameter at our disposal.  Then 
\begin{equation*}
 \sum_{j=J+1}^\infty \frac1j  \int_2^\infty k y^{k-1} j X^{-1} e^{-y^k j / X} \sum_{p\le y} e(j p^k \alpha)\, dy\ll  \sum_{j=J+1}^\infty \frac1j \left(\frac{X}{j}\right)^{1/k} \ll \left(\frac{X}{J}\right)^{1/k}.
\end{equation*}
In summary, for any $J\ge 1$ we have 
\begin{align} \label{truncation of Phi}
\Phi( \rho e(\alpha)) & = \sum_{j=1}^J \frac1j  \sum_p e^{-p^k j / X} e(j p^k \alpha) + O\left( \left(\frac{X}{J}\right)^{1/k}\right)\\ 
\nonumber & \sum_{j=1}^J \frac1j \int_2^\infty k y^{k-1} j X^{-1} e^{-y^k j / X} \sum_{p\le y} e(j p^k \alpha)\, dy + O\left( \left(\frac{X}{J}\right)^{1/k}\right).
\end{align}

We conclude this section with an estimate of $\Phi(\rho e(\alpha))$ on the minor arcs.  
\begin{lemma}\label{minor arcs lemma}
For $\alpha \in \mathfrak{m}$, 
\begin{equation*}
\Phi( \rho e(\alpha)) \ll X^{\frac1k}(\log X)^{-2+\varepsilon}.
\end{equation*}
\end{lemma}

\begin{proof}

Fix $j\le \sqrt{X}$ and consider 
\begin{equation}\label{minor arc reduction}
  \int_2^\infty k y^{k-1} j X^{-1} e^{-y^k j / X} \sum_{p\le y} e(j p^k \alpha)\, dy.
 \end{equation}
We use Dirichlet's theorem to choose $a\in \mathbb Z$, $q\in \mathbb N$, so that 
$$\left| j\alpha - \frac{a}{q}\right| \le q^{-1} X^{-1} (\log X)^{12k}.$$
Note that since $\alpha\in\mathfrak m$, we must have $q> j^{-1}(\log X)^{12k}$.
By Kawada and Wooley \cite{KawWoo2001}, we have 
\begin{equation*}\label{Weyl-type bound} %Maybe later add a section about these Weyl inequalities
\sum_{p\le x} e(jp^k \alpha) \ll x^{1-\eta(k) + \varepsilon} + \frac{q^{-\frac{1}{2k} + \varepsilon} x(\log x)^4}{\left(1+x^k\left| j\alpha - \frac{a}{q}\right|\right)^{1/2}},
\end{equation*} 
where $0<\eta(k)<\frac12$ is a constant that can be made explicit.  Recalling \eqref{integral for any lambda} and observing that
\begin{equation*}
\int_2^\infty y (\log y)^4 \left(k y^{k-1} j X^{-1} e^{-y^k j / X} \right)\, dy \ll  \left(\frac{X}{j}\right)^{1/k} \left(\log\left(\frac{X}{j}\right)\right)^{4},
\end{equation*}
we see that the expression in \eqref{minor arc reduction} is 
\begin{align*}
& \int_2^\infty k y^{k-1} j X^{-1} e^{-y^k j / X} \sum_{p\le y} e(j p^k \alpha)\, dy \ll  \left(\frac{X}{j}\right)^{\frac{1-\eta(k) + \varepsilon}{k}} + q^{-\frac{1}{2k}+\varepsilon}\left(\frac{X}{j}\right)^{1/k} \left(\log \left(\frac{X}{j}\right)\right)^{4}\\ &
\ll  \left(\frac{X}{j}\right)^{\frac{1-\eta(k) + \varepsilon}{k}} +  \left(\frac{(\log X)^{12k}}{j}\right)^{-\frac{1}{2k}+\varepsilon}\left(\frac{X}{j}\right)^{1/k} \left(\log\left(\frac{X}{j}\right)\right)^{4}.
\end{align*}
  Putting this into \eqref{truncation of Phi} with $J=\lfloor\sqrt X\rfloor$, we obtain
\begin{align*}
\Phi( \rho e(\alpha)) & =  \sum_{j=1}^J \frac1j  \int_2^\infty k y^{k-1} j X^{-1} e^{-y^k j / X} \sum_{p\le y} e(j p^k \alpha)\, dy+  O\left(\left(\frac{X}{J}\right)^{\frac{1}{k}}\right) \\ & 
%\ll \sum_{j=1}^J \frac1j  \left(  \left(\frac{X}{j}\right)^{\frac{1-\eta(k) + \varepsilon}{k}} + \left(\frac{(\log X)^{12k}}{j}\right)^{-\frac{1}{2k}+\varepsilon} \left(\frac{X}{j}\right)^{\frac{1}{k}}\left(\log\left(\frac{X}{j}\right)\right)^{4}\right) +   X^{\frac{1}{2k}} \\ &
\ll X^{\frac{1-\eta(k)+\varepsilon}{k}} \left(\sum_{j=1}^J j^{-(1+\frac{1+\varepsilon}{k} - \frac{\eta(k)}{k})}\right) +  X^{\frac1k}(\log X)^{4-6+\varepsilon}\left(\sum_{j=1}^J j^{-(1+\frac{1+\varepsilon}{k} - \frac{1}{2k})}\right) + X^{\frac{1}{2k}} \\ &
\ll X^{\frac1k}(\log X)^{-2+\varepsilon},
\end{align*}
as desired.
\end{proof}

%%%%%%%%%%----------------------------------------%%%%%%%%%%%%%%%%%%%%
\section{Proof of Theorem \ref{main result full asymp}: Major arc estimates} \label{major arcs sec}
In this section we investigate the behavior of $\Phi(\rho e(\alpha))$ for $\alpha \in \mathfrak{M}$.  The goal is to obtain an estimate for $\Re(\Phi(\rho e(\alpha))$ so that we can bound the integrand of \eqref{Cauchy integral} on the major arcs with $q>1$.

\begin{lemma} \label{major arc estimate lemma}Suppose that $A$ is a positive real number, and that $X$ satisfies $X>X_0(A)$. Suppose also that $\alpha, \beta \in \mathbb{R}$, $q \in \mathbb N$, and $a\in\mathbb Z$ are such that $(a,q) = 1,$ $q\le (\log X)^A$, $\beta = \alpha-\frac{a}{q}$, $|\beta| \le q^{-1} X^{-1} (\log X)^A$.  Then 
\begin{equation*}
\Phi(\rho e(\alpha)) =  \frac{X^{1/k}\Gamma(1/k)}{(1 - 2\pi i \beta X)^{1/k}\log X} \sum_{j \le\sqrt{X}} \frac{S_k^*(q_j,a_j)}{j^{1+1/k}\varphi(q_j)} +  O\left(\frac{X^{1/k}\log\log X}{(\log X)^2}\right).
\end{equation*}
where $q_j = q/(q,j)$, $a_j = aj/(q,j)$, and 
$$S_k^*(q,a) = \sum_{\substack{\ell=1\\(\ell,q)=1}}^q e\left(\frac{a\ell^k}{q}\right).$$
\end{lemma} 
\begin{proof}
Recalling \eqref{truncation of Phi}, we note that for any integer $J$ 
\begin{align}
& \Phi\left(\rho e\left(\frac{a}{q} + \beta\right)\right)  = \sum_{j =1}^J  \frac{1}{j} \sum_{p \text{ prime}}e\left(\frac{ajp^k}{q}\right) \exp\left(j p^k(2\pi i \beta - 1/X)\right) + O\left(\left(\frac{X}{J}\right)^{\frac{1}{k}}\right) \nonumber \\  
\label{major arc innermost sum}
  & = \sum_{j =1}^J  \frac{1}{j} \left(\sum_{\substack{\ell = 1\\ (\ell, q_j) = 1}}^{q_j} e\left(\frac{a_j\ell^k}{q_j}\right)\sum_{p \equiv \ell\bmod{q_j} }  \exp\left(j p^k(2\pi i \beta - 1/X)\right) + O(q_j^{\varepsilon})\right) + O\left(\left(\frac{X}{J}\right)^{\frac{1}{k}}\right). 
\end{align}
Let $J=\lfloor\sqrt X\rfloor$.  By Abel summation,
\begin{equation}\label{Abel result}
\sum_{p \equiv \ell\bmod{q_j} }  \exp\left(j p^k(2\pi i \beta - 1/X)\right)=   \int_2^\infty   \pi(t;q_j, \ell) \left(\frac{j}{X}- 2\pi i j\beta\right)  kt^{k-1}   \exp(j t^k(2\pi i \beta - 1/X)) \, dt.
\end{equation}

For $t>(\frac{X}{j})^{1/k}$, we apply the Siegel-Walfisz theorem to obtain
\begin{equation*}
 \pi(t;q_j, \ell) = \frac{Li(t)}{\varphi(q_j)} + O\left(t \exp\left(-c\sqrt{\log t}\right)\right) = \frac{Li(t)}{\varphi(q_j)} + O\left(t \exp\left(-c\sqrt{k\log(X/j)}\right)\right)
\end{equation*}
 for some constant $c$ depending only on $A$.  When $t\le (\frac{X}{j})^{1/k}$ we clearly have 
\begin{equation*}
 \pi(t;q_j, \ell) = \frac{Li(t)}{\varphi(q_j)} + O\left(\left(\frac{X}{j}\right)^{1/k} \exp\left(-c\sqrt{k\log(X/j)}\right)\right).
\end{equation*}
Thus the integral in \eqref{Abel result} is (via integration by parts)
\begin{equation*}
%& \int_2^\infty\frac{  Li(t) }{\varphi(q_j)}   \left(1- 2\pi i \beta X\right)  jXkt^{k-1}   \exp(-jX t^k(1- 2\pi i \beta X)) \, dt\\  
  \int_2^\infty \frac{\exp(j t^k(2\pi i \beta - 1/X))}{\varphi(q_j)\log t}\, dt  
 + O\left(\left(1+ |\beta|X \right) \left(\frac{X}{j}\right)^{1/k} \exp\left(-c\sqrt{k\log(X/j)}\right)\right). %, 
\end{equation*}

%\begin{align*}
%E_1& \ll  \exp\left(-c\sqrt{k\log(X/j)}\right)\left(1+ |\beta|X \right) \int_{(\frac{X}{j})^{1/k}}^\infty t \frac{j}{X}kt^{k-1}   \exp(-j t^k/X)\, dt \\
%  & \ll  \exp\left(-c\sqrt{k\log(X/j)}\right)\left(1+ |\beta|X \right) \left(\frac{X}{j}\right)^{1/k}\int_0^\infty  u^{1/k} e^{-u}\, du \\
%  & \ll \left(1+ |\beta|X \right)  \left(\frac{X}{j}\right)^{1/k}\exp\left(-c\sqrt{k\log(X/j)}\right),
%  \end{align*} 
% and

%$$E_2 \ll \left(1+ |\beta|X \right) \left(\frac{X}{j}\right)^{1/k} \exp\left(-c\sqrt{k\log(X/j)}\right).$$

Putting this into \eqref{major arc innermost sum}, we have
\begin{align}\label{Phi into S(q,a) sum}
& \left|\Phi\left(\rho e\left(\frac{a}{q} + \beta\right)\right) -  \sum_{j \le\sqrt{X}} \frac{S_k^*(q_j,a_j)}{j\varphi(q_j)} \int_2^\infty \frac{\exp(j t^k(2\pi i \beta - 1/X))}{\log t}\, dt \right|
\\ &  \ll \sum_{j \le\sqrt{X}}\frac{\varphi(q_j)}{j} \left(1+ |\beta|X \right)  \left(\frac{X}{j}\right)^{1/k}\exp\left(-c\sqrt{k\log(X/j)}\right)+ X^{1/2k} \nonumber
\\ & \ll  \varphi(q)\left(1+ |\beta|X \right)X^{1/k}\exp\left(-c\sqrt{\frac{k}{2}\log(X)}\right)\ll \frac{X^{1/k}\log\log X}{(\log X)^2}, \nonumber
\end{align}
since 
$\varphi(q)|\beta|X \le (\log X)^A$, and $\exp\left(-c\sqrt{\frac{k}{2}\log(X)}\right) \ll (\log X)^{-(A+2)}$.

It remains to evaluate the integral 
\begin{equation*}
\int_2^\infty \frac{\exp(j t^k(2\pi i \beta - 1/X))}{\log t}\, dt.
\end{equation*}
Note that the integrand has absolute value less than $1$ for all $t$.  Thus 
$$\int_2^{\left(\frac{X}{j}\right)^{1/k}\frac{1}{(\log X)^2}} \ \frac{\exp(j t^k(2\pi i \beta - 1/X))}{\log t}\, dt\ll \left(\frac{X}{j}\right)^{1/k}\frac{1}{(\log X)^2}.$$
Meanwhile, the contribution from $t\ge  \left(\frac{X}{j}\right)^{1/k}\log\log X$ is 
\begin{align*}
 & \le \int_{\left(\frac{X}{j}\right)^{1/k}\log\log X}^\infty \frac{\exp(-j t^k /X)}{\log t}\, dt  
%\ll \left(\frac{X}{j}\right)^{1/k}\frac{1}{(\log X)} \int_{\log\log X}^\infty e^{-u^k}\, du  \\ &
\ll  \left(\frac{X}{j}\right)^{1/k}\frac{1}{(\log X)} \int_{\log\log X}^\infty e^{-u}\, du 
=  \left(\frac{X}{j}\right)^{1/k}\frac{1}{(\log X)^2} .
\end{align*}
If $\left(\frac{X}{j}\right)^{1/k}(\log X)^{-2} \le t\le  \left(\frac{X}{j}\right)^{1/k}\log\log X,$
then we have 
$$\frac{1}{\log t} = \frac{k}{\log X} + O\left(\frac{\log j+ \log\log X}{(\log X)^2}\right).$$
Furthermore
$$\int_{\left(\frac{X}{j}\right)^{1/k}(\log X)^{-2}}^{\left(\frac{X}{j}\right)^{1/k}\log\log X} \exp(-j t^k/X)\, dt \le 
\int_0^\infty \exp(-j t^k/X)\, dt \ll \left(\frac{X}{j}\right)^{1/k}.$$
%and 
%$$ \sum_{j \le\sqrt{X}} \frac1j \left(\frac{X}{j}\right)^{1/k}\frac{\log j+ \log\log X}{(\log X)^2}  \ll\frac{X^{1/k} \log\log X}{(\log X)^2}.$$
So we have 
\begin{align*}
& \int_2^\infty \frac{\exp(j t^k(2\pi i \beta - 1/X))}{\log t}\, dt \\ &= \frac{k}{\log X} \int_0^\infty \exp\left(-t^k jX^{-1}(1-2\pi i \beta X)\right)\, dt +  O\left( \left(\frac{X}{j}\right)^{1/k}\frac{\log j+ \log\log X}{(\log X)^2}\right).
%\\ &
%= \frac{k}{\log X} \left(\frac{X}{j}\right)^{1/k}(1 - 2\pi i \beta X)^{-1/k} \frac{1}{k} \int_0^\infty u^{1/k - 1} e^{-u} \, du \\ &
%= \left(\frac{X}{j}\right)^{1/k}\frac{\Gamma(1/k)}{\log(X/j)(1 - 2\pi i \beta X)^{1/k}} .
\end{align*}
We make the substitution $z = \left(jX^{-1}(1-2\pi i \beta X)\right)^{1/k} t$.  Choose $\phi$ so that $|\phi|<1/2$ and 
$$\frac{1-2\pi i \beta X}{|1-2\pi i \beta X|} = e^{2\pi i\phi}.$$
We thus obtain 
$$z = \left(jX^{-1}|1-2\pi i \beta X|\right)^{1/k} e^{2\pi i\phi/k} t.$$
This gives
$$\int_0^\infty \exp\left(-t^k jX^{-1}(1-2\pi i \beta X)\right)\,dx =  \left(\frac{X}{j(1-2\pi i \beta X)}\right)^{1/k}\int_\mathcal{L} e^{-z^k}\,dz,$$
where $\mathcal{L}$ is the ray $\{z = t e^{2\pi i\phi/k}:  0\le t < \infty\}$.   By Cauchy's theorem, the integral here is equal to $\Gamma(\frac{k+1}{k})$. 
Inserting this into \eqref{Phi into S(q,a) sum} we obtain
\begin{equation*}
\Phi(\rho e(\alpha)) =  \frac{k\Gamma(\frac{k+1}{k})X^{1/k}}{(1 - 2\pi i \beta X)^{1/k}\log X} \sum_{j \le\sqrt{X}} \frac{S_k^*(q_j,a_j)}{j^{1+1/k}\varphi(q_j)} + O\left(\frac{X^{1/k}\log\log X}{(\log X)^2}\right).
\end{equation*}
The result follows upon noticing that $k\Gamma(\frac{k+1}{k}) = \Gamma(\frac{1}{k})$.
\end{proof}

Ultimately, we need to bound the real part of $\Phi(\rho e(\alpha))$ on the major arcs.  Lemma \ref{major arc estimate lemma} shows that an understanding of $S_k^*(q,a)$ is central to that goal.  In Section \ref{appendix}, we analyze $S_k^*(q,a)$ asymptotically and show there exists a constant $C_k$ such that $|S_k^*(q,a)|\le C_k q^{-\frac1k} \varphi(q),$ for all $(q,a) =1$.  From that theory and the following lemma, we can achieve the bound that we need.

\begin{lemma}\label{major arc real part with delta}
For all $q>2$, we have 
$$\left|\Re \left(\frac{ S_k^* (q,a)}{(1 - 2\pi i \beta X)^{1/k}}\right)\right| \le \left( 1 - \delta_k\right) \varphi(q),$$
where 
$$ \delta_k = \frac{\pi^2}{2^{2k+3} C_k^{2k}}$$ and 
$$C_k = \begin{cases} 128 &\mbox{if } k=2 \\
\prod_{p\le k^6} k & \mbox{if }k\ge3 \end{cases} .$$
\end{lemma}

\begin{proof}
 If $q>(2C_k)^k$, then by Proposition \ref{Uniform S bound prop} we have $|S_k^* (q,a)|\le \frac12 \varphi(q)$.  Thus we restrict our attention to $q\le (2C_k)^k$. For notational convenience, let $C=  (2C_k)^k$.   
 
  As in the proof of Lemma \ref{major arc estimate lemma}, we let $\phi$ be the real number satisfying $-\frac12<\phi< \frac12$ and 
\begin{equation*}
\frac{1 - 2\pi i \beta X}{|1 - 2\pi i \beta X|} = e^{2\pi i\phi}.
\end{equation*}
We also let $\Delta:= |1 - 2\pi i \beta X| = \sqrt{1+ 4\pi^2|\beta|^2 X^2}$.

We have 
\begin{equation*}
\frac{ S_k^* (q,a)}{(1 - 2\pi i \beta X)^{1/k}} = \Delta^{-1/k} e(-\phi/k)  \sum_{\substack{\ell = 1\\ (\ell, q) = 1}}^{q} e\left(\frac{a\ell^k}{q}\right) =  \Delta^{-1/k} \sum_{\substack{\ell = 1\\ (\ell, q) = 1}}^{q} e\left(\frac{a\ell^k}{q}-\frac{\phi}{k}\right).
\end{equation*}
Since $(a,q) = 1$, we have $(a\ell^k, q)$= 1.  Let $\Vert x\Vert$ denote the distance to the nearest integer.  We have
\begin{equation*}
\left\Vert \frac{a\ell^k}{q}\right\Vert \ge \frac1q \ge \frac1C \quad\text{ and }\quad
\left\Vert \frac{a\ell^k}{q} - \frac12 \right\Vert \ge \frac{1}{2q} \ge \frac{1}{2C}.
\end{equation*}
It follows that each term $e(\frac{a\ell^k}{q})$ is on the unit circle with argument bounded away from $0$ and $\pi$.  The main idea is to show that the terms of $S_k^*(q,a)$ stay bounded away from $\pm 1$ when multiplied by $e(-\phi/k)\Delta^{-1/k}$.
  
 If $|\phi| < \frac{k}{4C}$, then $e(-i\phi)$ doesn't twist the terms of $S_k^*(q,a)$ much and so arguments are still bounded away from $0$ and $\pi$.  More precisely, %Thus 
%\begin{equation*}
%\left\Vert \frac{a\ell^k}{q}-\frac{\phi}{k} \right\Vert > \frac{1}{2C} -  \frac{1}{4C} =  \frac{1}{4C} 
%\quad\text{ and }\quad
%\left\Vert \left(\frac{a\ell^k}{q}-\frac{\phi}{k}\right) - \frac12 \right\Vert > \frac{1}{2C} -  \frac{1}{4C} =  \frac{1}{4C}.
%\end{equation*}
%So, 
\begin{equation*}
\left|\Re \left(e\left(\frac{a\ell^k}{q}-\frac{\phi}{k}\right)\right)\right| = \left|\cos\left(2\pi\left(\frac{a\ell^k}{q}-\frac{\phi}{k}\right)\right) \right| < \cos\left( \frac{\pi}{2C}\right) < 1 - \frac{2}{5} \left( \frac{\pi}{2C}\right)^2,
\end{equation*}
since $\frac25 \le (1-\cos\omega)\omega^{-2} < \frac12$ for $0<\omega<\pi/2$.
Hence 
\begin{equation*}
\left|\Re \left(\frac{ S_k^* (q,a)}{(1 - 2\pi i \beta X)^{1/k}}\right)\right| < \Delta^{-1/k} \left(1- \frac{\pi^2}{10C^2}\right) \varphi(q) < \left(1- \frac{\pi^2}{10C^2}\right) \varphi(q)
\end{equation*}
%If $\phi$ is larger, then ${|1 - 2\pi i \beta X|}$ is larger, so $|S(q,a)|$ is scaled down.
%If  $|\phi| < \frac{k}{4C}$, then

Now consider $|\phi| \ge  \frac{k}{4C}$.  Since $2\pi \phi = \arg(1-i2\pi\beta X)$, we have that 
\begin{equation*}
2\pi |\beta| X = \tan |2\pi\phi| \ge \tan\left(\frac{\pi k }{2C}\right).
\end{equation*}
So
\begin{equation*}
\Delta = \left|1- i (2\pi \beta X)\right| \ge \left|1 - i\tan\left(\frac{\pi k }{2C}\right)\right| = \sqrt{1 + \tan^2\left(\frac{\pi k }{2C}\right)}.
\end{equation*}
By the binomial theorem and the fact that $\tan^2 y>y^2$ for $|y|\le \pi/2$, we thus have
$$\Delta^{-1/k} \le \left(1 + \tan^2\left(\frac{\pi k }{2C}\right)\right)^{-\frac{1}{2k}}\le 1 - \frac{1}{2k}\left(\frac{\pi k }{2C}\right)^2.$$
So for $|\phi| \ge \frac{k}{4C}$,
\begin{equation*}
\left|\Re \left(\frac{ S_k^* (q,a)}{(1 - 2\pi i \beta X)^{1/k}}\right)\right| \le  \Delta^{-1/k} \left|S_k^* (q,a)\right| \le \left( 1 - \frac{\pi^2 k }{8C^2}\right)\varphi(q).
\end{equation*}
%Let 
%\begin{equation*}
%\delta_{C,k} = \max \left(\left|\cos\left( \frac{\pi}{2C}\right)\right|, \left(1 + \tan^2\left(\frac{\pi k }{2C}\right)\right)^{-\frac{1}{2k}}\right).
%\end{equation*}
In either case,  
\begin{equation*}
\left|\Re \left(\frac{ S_k^* (q,a)}{(1 - 2\pi i \beta X)^{1/k}}\right)\right| \le \left( 1 - \frac{\pi^2}{8C^2}\right)\varphi(q) = \left( 1 - \delta_k\right)\varphi(q).
\end{equation*}
\end{proof}

We are now ready to complete the analysis of the non-principal major arcs.  
\begin{lemma} \label{major arc real part lemma} Suppose that $A$ is a positive real number, and that $X$ satisfies $X>X_0(A)$. Suppose also that $\alpha, \beta \in \mathbb{R}$, $q \in \mathbb N$, and $a\in\mathbb Z$ are such that $(a,q) = 1,$ $1 < q\le (\log X)^A$, $\beta = \alpha-\frac{a}{q}$, $|\beta| \le q^{-1} X^{-1} (\log X)^A$.  Then
\begin{equation*}
\left|\Re \left(\Phi(\rho e(\alpha))\right)\right| \le  \left(1-\frac{\delta_k}{3} \right)\Phi(\rho) \left(1+ O\left(\frac{\log\log X}{\log X}\right)\right),
\end{equation*}
where $\delta_k$ is the constant defined in Lemma  \ref{major arc real part with delta}.
\end{lemma}

\begin{proof}
%As above, we let $\phi$ and $\Delta$ be the real numbers satisfying
%$$ 1 - 2\pi i \beta X = \Delta e^{2\pi i \phi},\quad -\frac12 <\phi <\frac12.$$
 By Lemma \ref{major arc estimate lemma}, we have
  %\begin{equation*}
%\Re\left(\Phi(\rho e(\alpha))\right) <  \delta_{Q,k}\frac{X^{1/k}\Gamma(1/k)\zeta(1+1/k)}{\log X}  + O\left(\frac{X^{1/k}\log\log X}{(\log X)^2}\right) < \delta_{Q,k} \Phi(\rho).
%\end{equation*}
\begin{equation}\label{real part starting point}
\Re\left(\Phi(\rho e(\alpha)) \right)=\frac{X^{1/k}\Gamma(1/k)}{\log X}\Re\left(\frac{1}{(1 - 2\pi i \beta X)^{1/k}} \sum_{j \le\sqrt{X}}\frac{S_k^*(q_j,a_j)}{j^{\frac{k+1}{k}}\varphi(q_j)}\right) +  O\left(\frac{X^{1/k}\log\log X}{(\log X)^2}\right)
\end{equation}

We first handle the case $q=2$.  Then $q_j = 1,2$ depending on whether $j$ is even or odd, respectively. So $S_k^*(q_j, a_j) = (-1)^j \varphi(q_j)$.  Hence
\begin{align*}
\left| \sum_{j \le\sqrt{X}}\frac{S_k^*(q_j,a_j)}{j^{\frac{k+1}{k}}\varphi(q_j)}\right| & =  \left|\sum_{j \le\sqrt{X}}\frac{(-1)^j}{j^{\frac{k+1}{k}}} \right| \le \sum_{\substack{j \le\sqrt{X}\\j \text{ odd}}}\frac{1}{j^{\frac{k+1}{k}}} \\ &
\le \left(1- 2^{-\frac{k+1}{k}}\right)\zeta \left(\frac{k+1}{k}\right) < \left(1- \frac{\delta_k}{3}\right)\zeta \left(\frac{k+1}{k}\right) .
\end{align*}

Now consider $q>2.$  % If $q\mid 2j$, then $q_j = \frac{q}{(q,j)} \in \{1,2\}$ hence $S_k^*(q_j, a_j) = \pm \varphi(q_j) $.  %Note that for odd $q$, $q\mid 2j$ if and only if $q\mid j$ and for even $q$, $q\mid 2j$ if and only if $\frac q2\mid j$.
 %On the other hand, 
 If $q\nmid 2j$,  then $q_j>2$ and Lemma \ref{major arc real part with delta} tells us that 
 $$\left|\Re\left(\frac{S_k^*(q_j,a_j)}{(1 - 2\pi i \beta X)^{1/k}}\right) \right| \le \left(1-\delta_k\right) \varphi(q_j).$$
  Thus% the sum in \eqref{real part starting point} satisfies 
 \begin{align*}
 &\left| \Re\left(\frac{1}{(1 - 2\pi i \beta X)^{1/k}} \sum_{j \le\sqrt{X}}\frac{S_k^*(q_j,a_j)}{j^{\frac{k+1}{k}}\varphi(q_j)}\right)\right| 
\le \sum_{\substack{j\le \sqrt{X}\\ q\mid 2j}}\frac{1}{j^{\frac{k+1}{k}}}  +  \left(1-\delta_k\right) \sum_{\substack{j\le \sqrt{X}\\ q\, \nmid\, 2j}}\frac{1}{j^{\frac{k+1}{k}}} \\ &
% \le \sum_{m=1}^\infty\frac{1}{(2qm)^{\frac{k+1}{k}}} +  \left(1-\delta_k \right) \sum_{\substack{j=1\\ q\nmid 2j}}^\infty\frac{1}{j^{\frac{k+1}{k}}} \\ &
\le  \left(\frac{q}{2}\right)^{-\frac{k+1}{k}}\zeta\left(\frac{k+1}{k}\right) + \left(1-\delta_k \right)\left(1-\left(\frac{q}{2}\right)^{-\frac{k+1}{k}}\right) \zeta\left(\frac{k+1}{k}\right) \\ & 
\le  \left(1-\frac{\delta_k}{3} \right)\zeta\left(\frac{k+1}{k}\right) .
 \end{align*} 
 
Putting this into \eqref{real part starting point}, we have (for any $q\ge 2$),
 \begin{align*}\left| \Re\left(\Phi(\rho e(\alpha)) \right)\right| & \le \frac{ \left(1-\frac{\delta_k}{3} \right)  \zeta\left(\frac{k+1}{k}\right) \Gamma(1/k)X^{1/k}}{\log X} +  O\left(\frac{X^{1/k}\log\log X}{(\log X)^2}\right) \\ & \le
  \left(1-\frac{\delta_k}{3} \right)\Phi(\rho) \left(1+ O\left(\frac{\log\log X}{\log X}\right)\right).
 \end{align*}

\end{proof}

%%%%%%%%%%----------------------------------------%%%%%%%%%%%%%%%%%%%%
\section{Proof of Theorem \ref{main result full asymp}: The main term}\label{end of proof sec}
In this section, we conclude the proof of Theorem \ref{main result full asymp} by analyzing the contribution from $\mathfrak M(1,0)$.  
Recall that $p_{\mathbb P_k}(n)$ is given by the integral in \eqref{Cauchy integral}.  
For $\alpha\not\in  \mathfrak{M}(1,0)$ and $n$ sufficiently large, we have
\begin{equation*}
\left|\Re\left(\Phi(\rho e(\alpha))\right)\right| \le \left(1-\frac{\delta_k}{3}\right) \Phi(\rho),
\end{equation*}  
by Lemma \ref{minor arcs lemma} and Lemma \ref{major arc real part lemma}.  Thus 
\begin{equation}\label{integral principal major arc}
p_{\mathbb P_k}(n) = \rho^{-n} \int_{\mathfrak M(1,0)} \exp\left(\Phi(\rho e(\alpha)) \right)e(-n\alpha) \, d\alpha + O\left(\rho^{-n} \Psi(\rho) n^{-B}\right),
\end{equation}
for any constant $B$.  

If $\alpha \in \mathfrak{M}(1,0)$ and $|\alpha| > X^{-1} (\log X)^{-1/4}$, then by Lemma \ref{major arc estimate lemma} we have
\begin{equation*}
\left|\Phi(\rho e(\alpha))\right| =  \frac{X^{1/k}\Gamma(1/k)\zeta(1+1/k)}{(1 + 4\pi^2 \alpha^2 X^2)^{1/2k}\log X} +  O\left(\frac{X^{1/k}\log\log X}{(\log X)^2}\right),
\end{equation*}
and by the binomial theorem
\begin{equation*}
\left(1 + 4\pi^2 \alpha^2 X^2\right)^{-1/2k}< \left(1-\frac{2\pi^2}{k} \alpha^2 X^2\right) < 1-\frac{2\pi^2}{k} (\log X)^{-1/2}.
\end{equation*}
Thus for $X$ sufficiently large, we have
\begin{equation*}
\left|\Re\Phi(\rho e(\alpha))\right| < \left(1-(\log X)^{-1}\right)\Phi(\rho).
\end{equation*}
The contribution to the integral in \eqref{integral principal major arc} from such $\alpha$ is then
\begin{equation}\label{reduce to tau}
\le  \rho^{-n} \Psi(\rho) \exp\left(-\frac{\Phi(\rho)}{\log X}\right)\ll \rho^{-n} \Psi(\rho) n^{-B}.
\end{equation}

\begin{comment}
The goal of this section is to estimate the main term for $p_{\mathbb P_k}(n)$.  We follow the proof of Theorem 2 from Vaughan \cite{Vaughan2008}.  For $n$ large, we recall $\rho = \rho(n)$ and $X= X(n) $ are defined by $n= \rho \Phi'(\rho)$ and $\rho = e^{-1/X}$.  It follows from  \eqref{x log(1/rho)}  that 
\begin{equation*}
 X = \left(\log\left(\frac1\rho\right)\right)^{-1} \sim \left( \frac{k}{k+1} \frac{n\log n}{\zeta\left(1+\frac1k\right)\Gamma\left(1+\frac1k\right)}\right)^{\frac{k}{k+1}}.
\end{equation*}
\end{comment}

We are left to study the integral
\begin{equation} \label{main term tau interval}
\rho^{-n} \int_{-\tau}^{\tau} \exp(\Phi(\rho e(\alpha))) e(-n\alpha)\, d\alpha,
\end{equation}
where
\begin{equation*} \tau = X^{-1}(\log X)^{-1/4}.
\end{equation*}
For $\alpha\in\mathbb{R}$, let $R(\alpha)$ and $I(\alpha)$ denote the real and imaginary parts of $\Phi(\rho e(\alpha))$, respectively.  Applying Taylor's theorem to each of $R(\alpha)$ and $I(\alpha)$, we have 
\begin{align}\label{taylor expansion Phi}
\Phi(\rho e(\alpha)) & = R(\alpha)+ i I(\alpha) \\ & \nonumber
= (R(0) +i I(0)) + \alpha (R'(0) +iI'(0))+ \frac{\alpha^2}{2}  (R''(0) +iI''(0)) \\ & \nonumber
+ \frac{\alpha^3}{6} (R'''(c_R \alpha) +i I'''(c_I \alpha)),
\end{align}
where $0<c_R, c_I <1$ may depend on $\alpha$. 

Now, for any real $\beta$ we have
\begin{equation*}
R'(\beta) + i I'(\beta) = \frac{d}{d\beta} \Phi(\rho e(\beta))  = 2\pi i e(\beta) \rho \Phi'(\rho e(\beta)),
\end{equation*}
\begin{equation*}
R''(\beta) + i I''(\beta) = -4\pi^2  e(\beta) \rho \Phi'(\rho e(\beta))  - 4\pi^2  e(2\beta) \rho^2 \Phi''(\rho e(\beta)),
\end{equation*}
and
\begin{align*}
R'''(\beta) + i I'''(\beta) & = -8\pi^3 i e(\beta) \rho \Phi'(\rho e(\beta)) - 24\pi^3 i e(2\beta) \rho^2 \Phi''(\rho e(\beta)) \\ &
- 8\pi^3 i e(3\beta) \rho^3 \Phi'''(\rho e(\beta)).
\end{align*}
Thus 
\begin{equation*}
\sup\left(|R'''(\beta)|, |I'''(\beta)|\right) \le 8\pi^3\left( \rho \Phi'(\rho) +3\rho^2 \Phi''(\rho) +  \rho^3 \Phi'''(\rho)\right).
\end{equation*}
Hence, there exists $w\in\mathbb C$ (depending on $\alpha$) such that $|w|\le 1$ and  
\begin{equation*}
\alpha^3 \left(R'''(c_R\alpha) + i I'''(c_I\alpha)\right) = 16\pi^3w|\alpha|^3 \left( \rho \Phi'(\rho) +3\rho^2 \Phi''(\rho) +  \rho^3 \Phi'''(\rho)\right).
\end{equation*}
Putting this into \eqref{taylor expansion Phi}, we have
\begin{align*}
\Phi(\rho e(\alpha)) & =  \Phi(\rho ) + 2\pi i \alpha\rho \Phi'(\rho) 
 - 2\pi^2\alpha^2 ( \rho \Phi'(\rho) + \rho^2 \Phi''(\rho))  \\ &
+\frac{8}{3}\pi^3w|\alpha|^3 \left( \rho \Phi'(\rho) +3\rho^2 \Phi''(\rho) +  \rho^3 \Phi'''(\rho)\right).
\end{align*}
Recalling that $n= \rho \Phi'(\rho)$, the integral in \eqref{main term tau interval} becomes 
\begin{equation*}
\rho^{-n} \int_{-\tau}^{\tau}\exp(g(\rho, \alpha))\, d\alpha,
\end{equation*}
where
\begin{equation}\label{exponential integrand}
g(\rho, \alpha) =  \Phi(\rho ) - 2\pi^2\alpha^2 ( \rho \Phi'(\rho) + \rho^2 \Phi''(\rho))  
+\frac{8}{3}\pi^3w|\alpha|^3 \left( \rho \Phi'(\rho) +3\rho^2 \Phi''(\rho) +  \rho^3 \Phi'''(\rho)\right).
\end{equation}

The next step is to further narrow the interval of integration contributing to the main term.  We do this by showing that the real part of $g(\rho,\alpha)$ is strictly less than $\Phi(\rho)$ when $\alpha$ is not too small.  By Lemma \ref{rho d/drho lemma}, 
\begin{equation*}
\rho \Phi'(\rho) + \rho^2 \Phi''(\rho) \gg X^{2+\frac1k}(\log X)^{-1},
\end{equation*}
and 
\begin{equation*}
 \rho \Phi'(\rho) +3\rho^2 \Phi''(\rho) +  \rho^3 \Phi'''(\rho) \ll X^{3+\frac1k}(\log X)^{-1}.
\end{equation*}
Thus there exist positive constants $C_1, C_2$, such that for $X$ sufficiently large and  $|\alpha| \le \tau = X^{-1}(\log X)^{-1/4}$ we have
\begin{align*}
& \left|\frac{8}{3}\pi^3w|\alpha|^3 \left( \rho \Phi'(\rho) +3\rho^2 \Phi''(\rho) +  \rho^3 \Phi'''(\rho)\right)\right|  \\ &
\le C_1 \alpha^2 X^{2+\frac1k}(\log X)^{-5/4} \le C_2 \alpha^2 X^{2+\frac1k}(\log X)^{-1} \\ & 
\le \pi^2 \alpha^2 ( \rho \Phi'(\rho) + \rho^2 \Phi''(\rho)).
\end{align*}
So, 
\begin{equation*}
|\Re( g(\rho, \alpha))| \le  \Phi(\rho ) - \pi^2\alpha^2 ( \rho \Phi'(\rho) + \rho^2 \Phi''(\rho)).
\end{equation*}
For $|\alpha| \ge  X^{-(1+\frac{1}{2k})} (\log X)^{2}$, there is a positive constant $C_3$ such that
 \begin{equation*}
|\Re( g(\rho, \alpha))| \le  \Phi(\rho ) - C_3 (\log X)^3.
\end{equation*}
Hence the contribution to the integral in \eqref{integral principal major arc} from these $\alpha$ is
 \begin{equation}\label{reduce to eta}
\le  \rho^{-n} \Psi(\rho) X^{-C_3(\log X)^2}\ll \rho^{-n} \Psi(\rho) n^{-B}.
\end{equation}

Finally, we need to estimate the integral over the interval $[-\eta, \eta]$, where $\eta = X^{-(1+\frac{1}{2k})} (\log X)^{2}.$
For $\alpha$ in this interval, we have
\begin{align*}
&|\alpha|^3( \pi^3\rho \Phi'(\rho) + 3\pi^3\rho^2 \Phi''(\rho) +  \pi^3 \rho^3 \Phi'''(\rho) )\\ &
\ll (\log X)^6 X^{-(3+\frac3{2k})} (\log X)^{-1} X^{3+\frac1k} =  (\log X)^5 X^{-\frac{1}{2k}}.
\end{align*}
Recall that $n = x \asymp X^{\frac{k+1}{k}}(\log X)^{-1}$.
So,
\begin{align*}
 X^{\frac{1}{2k}} (\log X)^{-5}& =   \left(X^{\frac{k+1}{k}}\right)^{\frac{1}{2(k+1)} }(\log X)^{-5}
\\ & \gg n^{\frac{1}{2(k+1)}}(\log X)^{-5 +\frac{1}{2(k+1)}}  \gg n^{\frac{1}{2(k+1)}+\varepsilon}
\\ & \gg n^{\frac{1}{2k+3}}.
\end{align*}
Hence 
\begin{equation*}
|\alpha|^3( \pi^3\rho \Phi'(\rho) + 3\pi^3\rho^2 \Phi''(\rho) +  \pi^3 \rho^3 \Phi'''(\rho))
\ll n^{-\frac{1}{2k+3}},
\end{equation*}
and therefore 
\begin{equation*}
\exp\left(\frac{8}{3}\pi^3w|\alpha|^3 \left( \rho \Phi'(\rho) +3\rho^2 \Phi''(\rho) +  \rho^3 \Phi'''(\rho)\right)\right) = 1+ O(n^{-\frac{1}{2k+3}})
\end{equation*}
Putting this into \eqref{exponential integrand}, and noticing that $ \Phi_{(2)}(\rho) =  \rho \Phi'(\rho) + \rho^2 \Phi''(\rho)$, we thus have
\begin{align*}
& \int_{-\eta}^{\eta} \exp(\Phi(\rho e(\alpha))) e(-n\alpha)\, d\alpha =  \int_{-\eta}^{\eta} \exp(g(\rho, \alpha))\, d\alpha
\\ & = (1+ O(n^{-\frac{1}{2k+3}})) \Psi(\rho) \int_{-\eta}^{\eta} \exp(-\alpha^2 2\pi^2 \Phi_{(2)}(\rho) )\, d\alpha.  
\end{align*}
Recall that $\eta  \Phi_{(2)}(\rho) \gg (\log X)^3$.  Through a routine polar coordinates integration, we have
\begin{align*}
  \left( \int_{-\eta}^{\eta} \exp(-\alpha^2 2\pi^2 \Phi_{(2)}(\rho) )\, d\alpha \right)^2 &
% = \left( \iint_{u^2 + v^2 \le \eta^2} \exp(-(u^2+v^2) 2\pi^2 \Phi_{(2)}(\rho) )\, d\alpha  \right)^2 \\&
%= \int_0^{2\pi} \int_0^\eta \exp(-\alpha^2 2\pi^2 \Phi_{(2)}(\rho) )\, r\, dr\, d\theta 
= \frac{1}{2\pi \Phi_{(2)}(\rho) } \left(1-\exp(-\eta^2 2\pi^2  \Phi_{(2)}(\rho) )\right)\\ &
 = \frac{1}{2\pi \Phi_{(2)}(\rho) } \left(1+O\left(e^{-(\log X)^2}\right )\right) .
\end{align*}
Therefore
\begin{equation*}
 \int_{-\eta}^{\eta} \exp(\Phi(\rho e(\alpha))) e(-n\alpha)\, d\alpha 
  = \frac{\Psi(\rho)}{\sqrt{2\pi\Phi_{(2)}(\rho)}} \left(1+ O(n^{-\frac{1}{2k+3}})\right)  \left(1+O\left(e^{-(\log X)^2}\right )\right).
\end{equation*}
Combining this with \eqref{integral principal major arc}, \eqref{reduce to tau}, and \eqref{reduce to eta} gives the desired result.

\begin{comment} \ayla{this is old... from when we used an infinite integral instead of polar coordinates}
Notice that 
\begin{equation*}
\frac{\log X}{X^{2+\frac1k}} \sim \frac{1}{nX},
\end{equation*}
and recall
\begin{equation*}
\frac{1}{X} \sim \left(\frac{k}{k+1} \frac{n\log n}{\zeta\left(1+\frac1k\right)\Gamma\left(1+\frac1k\right)} \right)^{-\frac{k}{k+1}}.
\end{equation*}
So the integral in \eqref{infinite integral} is 
\begin{equation*}
\sim \frac{1}{\sqrt{2\pi\zeta\left(1+\frac1k\right)\Gamma\left(2+\frac1k\right)}} n^{-1/2} \left(\frac{k}{k+1} \frac{n\log n}{\zeta\left(1+\frac1k\right)\Gamma\left(1+\frac1k\right)} \right)^{-\frac{k}{2(k+1)}}.
\end{equation*}
Comparing this to \eqref{Phi_m estimate}, (\ayla{and assuming the error terms behave}) we obtain
\begin{equation*}
p_{\mathbb P_k} (n) = \frac{\rho^{-n} \Psi(\rho)}{\sqrt{2\pi \Phi_{(2)}(\rho)}} \left(1+ O(n^{-\frac{1}{2k+3}})\right),
\end{equation*}
as desired.
\end{comment}

%%%%%%-------------------------------------------%%%%%%%%%%%%%%%

\section{Proof of Theorem \ref{difference theorem}} \label{difference sec}
Let $\rho = \rho(n)$ and let $X$ satisfy $\rho = e^{-1/X}$.  By \eqref{Cauchy integral}, we have
$$p_{\mathbb P_k} (n+1) - p_{\mathbb P_k} (n) = \int_{-1/2}^{1/2}  \rho^{-n} \exp(\Phi (\rho e(\alpha))-2\pi i n \alpha)(\rho^{-1}e^{-2\pi i \alpha} -1)\, d\alpha.$$
Note that $|\rho^{-1}e^{-2\pi i \alpha} -1| \le e^{1/X}+1\le 4$.  From the proof of Theorem \ref{main result full asymp}, we see that the contribution from $|\alpha|> \eta = X^{-(1+\frac{1}{2k})} (\log X)^{2}$ is 
$$\ll   \frac{\rho^{-n} \Psi(\rho)}{\sqrt{2\pi \Phi_{(2)}(\rho)}} n^{-B}$$
for any positive constant $B$.  On the other hand, when $|\alpha|\le\eta$, we have
$$\rho^{-1}e^{-2\pi i \alpha} -1 = \exp\left(\frac{1}{X} - 2\pi i \alpha\right) - 1 = \frac{1}{X} +O(\eta) = \frac1X + O\left(X^{-(1+\frac{1}{2k})} (\log X)^{2}\right).$$
From the proof of Theorem \ref{main result full asymp}, we have
\begin{equation*}\int_{-\eta}^{\eta} \rho^{-n} \exp(\Phi (\rho e(\alpha))-2\pi i n \alpha)\, d\alpha 
  = \frac{\rho^{-n}\Psi(\rho)}{\sqrt{2\pi\Phi_{(2)}(\rho)}} (1+ O(n^{-\frac{1}{2k+3}})).
\end{equation*}
Recalling \eqref{X in terms of x}, we thus have
\begin{align*}
p_{\mathbb P_k} (n+1) - p_{\mathbb P_k} (n) & =  \frac{\rho^{-n}\Psi(\rho)}{\sqrt{2\pi\Phi_{(2)}(\rho)}} (1+ O(n^{-\frac{1}{2k+3}}))\left( \frac1X + O\left(X^{-(1+\frac{1}{2k})} (\log X)^{2}\right)\right)\\ &
 =  \frac{\rho^{-n}\log(\frac{1}{\rho})\Psi(\rho)}{\sqrt{2\pi\Phi_{(2)}(\rho)}} (1+ O(n^{-\frac{1}{2k+3}})),
\end{align*}
as desired. \hfill$\square$
%\sim   \left(\left(1+\frac1k\right)\zeta\left(1+\frac1k\right)\Gamma\left(1+\frac1k\right)\right)^{\frac{k}{k+1}} \frac{ p_{\mathbb P_k} (n) }{(n\log n)^{\frac{k}{k+1}}}.

%%%%%%%---------------------------
\section{Analysis of $S_k^*(q,a)$}\label{appendix}

In Section \ref{major arcs sec}, we introduced the exponential sum $S_k^*(q,a)$.  Here we provide some analysis of these exponential sums.  For any $a,q\in \mathbb{N}$, we define
$$S_k(q,a)  =\sum_{\ell = 1}^q e\left(\frac{a\ell^k}{q}\right), \quad S_k^*(q,a)  =\sum_{\substack{\ell = 1\\ (\ell,q)=1}}^q e\left(\frac{a\ell^k}{q}\right).$$
Note that
$$ S_k^*(q,a) = \sum_{\nu\mid q}  \mu(\nu)S_k\left(\frac{q}{\nu},a\nu^{k-1}\right).$$

We begin by showing that $S_k^*(q,a)$ exhibits a certain multiplicative behavior.  We then show that  $S_k^*(p^\ell,a) = 0$ for large powers $\ell$.  Finally, we obain an effective upper bound for $ S_k^*(q,a)$, which is provided by Proposition \ref{Uniform S bound prop}.  The arguments here are similar to the treatment of  $S_k(q,a)$ in \cite[Ch. 4]{Vaughan1997}.

\begin{lemma}\label{S multiplicative}
 If $(q,r) = (a,r)  = (a,q) = 1$, then 
 $$S_k^*(qr,a) =  S_k^*(q,ar^{k-1}) S_k^*(r,aq^{k-1}).$$
\end{lemma}
\begin{proof}
We have 
\begin{align}
S_k^*(q,ar^{k-1}) S_k^*(r,aq^{k-1}) & = \sum_{\nu\mid q}  \sum_{\eta\mid r} \mu(\nu\eta) S_k\left(\frac{q}{\nu},ar^{k-1}\nu^{k-1}\right) S_k\left(\frac{r}{\eta},aq^{k-1}\eta^{k-1}\right) \nonumber \\ & 
=  \sum_{\nu\mid q}  \sum_{\eta\mid r} \mu(\nu\eta) \sum_{m\le \frac{q}{\nu}} \sum_{\ell\le \frac{r}{\eta}} e\left(a\frac{(r\nu m)^k + (q\eta\ell)^k}{qr}\right). \label{stop for euclid}
\end{align}
By Euclid's algorithm, for each residue class $h \pmod{\frac{qr}{\nu\eta}}$, there exists a unique pair $(m,\ell)$ with $m\le \frac{q}{\nu}$ and $\ell\le \frac{r}{\eta}$ such that 
$$h = \frac{r}{\eta} m + \frac{q}{\nu}\ell.$$
Let $\lambda = \nu\eta$.  Then $h\lambda = r\nu m + q\eta\ell$, and \eqref{stop for euclid} is equal to
\begin{equation*}
 \sum_{\lambda\mid qr}   \mu(\lambda) \sum_{h\le \frac{qr}{\lambda}} e\left(\frac{a(h\lambda)^k}{qr}\right) = S_k^*(qr,a).
\end{equation*}
\end{proof}

\begin{lemma} \label{S often zero}
For each prime $p$, define $\tau = \tau(p)$ so that $p^\tau || k$, and define 
$$\gamma = \gamma(p) = \begin{cases} \tau + 2, & \text{if } p=2 \text{ and } \tau > 0 \\
\tau + 1, & \text{otherwise.}\end{cases}$$
Then $S_k^*(p^\ell, a) = 0$ whenever $\ell>\gamma$ and $p\nmid a$.
\end{lemma}

It is useful to note that $\gamma\le k$ unless $p=k=2$, in which case $\gamma =3$.

\begin{proof}
We have
\begin{equation*}
 S_k^*(p^\ell,a) =  S_k(p^\ell,a)  - S_k(p^{\ell-1},ap^{k-1}).
 \end{equation*}
If $\ell \le k$, then $S_k(p^\ell,a) = p^{\ell-1}$ by Lemma 4.4 of \cite{Vaughan1997} and 
$$S_k\left(p^{\ell-1},ap^{k-1}\right) = \sum_{m\le p^{\ell-1}} e\left(\frac{ap^km^k}{p^\ell}\right) = p^{\ell-1}.$$
If $\ell > k$ then $S_k(p^\ell,a) = p^{k-1} S_k(p^{\ell-k},a)$ by Lemma 4.4 of \cite{Vaughan1997} and
 $$S_k\left(p^{\ell-1},ap^{k-1}\right) = \sum_{m\le p^{\ell-1}} e\left(\frac{am^k}{p^{\ell-k}}\right) =  p^{k-1}\sum_{r\le p^{\ell-k}} e\left(\frac{ar^k}{p^{\ell-k}}\right)= p^{k-1}S_k\left(p^{\ell-k},a\right).$$
\end{proof} 

\begin{proposition}\label{Uniform S bound prop}
For all $q,a\in \mathbb N$ with $(q,a)= 1$, we have
$$ |S_2^*(q,a)| \le 8 q^{-1/4}\varphi(q),$$
and for $k>2$
$$|S_k^*(q,a)|\le C_k q^{-\frac1k} \varphi(q),$$
where
$$C_k =  \prod_{p\le k^6} k$$
\end{proposition}

\begin{proof}
Write $q= \prod p^\ell.$  Then by Lemma \ref{S multiplicative}, we have
$$S_k^*(q,a) = \prod_{p\mid q} S_k^*\left(p^\ell, a(qp^{-\ell})^{k-1}\right).$$

We first consider $k=2$.  If $p^2\mid q$ for $p\ge3$, or if $16|q$, then $S_2^*(q,a) = 0$ by Lemma \ref{S often zero}.  So we may write $q= 2^\ell b$, where $0\le \ell\le 3$ and $b$ is odd, squarefree.  

It is easy to see that $|S_2^*(2^\ell,a)| = 2^{\ell-1}$ for $1\le\ell\le 3$. For odd primes $p$
\begin{align*}
|S_2^*(p,a)|^2 & = \sum_{x=1}^{p-1}\sum_{y=1}^p e\left(\frac{a(y^2-x^2)}{p}\right) -  \sum_{x=1}^{p-1}e\left(\frac{-ax^2}{p}\right) \\ &
= \sum_{x=1}^{p-1}\sum_{h=1}^p e\left(\frac{a(2x+h)h}{p}\right)  - S_2^*(p,-a) \\ &
= \sum_{h=1}^p e\left(\frac{ah^2}{p}\right) \sum_{x=1}^{p-1}e\left(\frac{2ahx}{p}\right)  - S_2^*(p,-a )\\ &
= \sum_{h=1}^{p-1} (-1)e\left(\frac{ah^2}{p}\right) + (p-1)  - S_2^*(p,-a) \\ &
= (p-1) -S_2^*(p,a) - S_2^*(p,-a).
\end{align*}
Thus 
\begin{equation*}
|S_2^*(p,a)| \le \sqrt p + 1 = p^{-\frac14} (p-1) \left(\frac{ p^{\frac14}}{\sqrt{p}-1}\right) < \begin{cases} p^{-\frac14} (p-1), & p\ge 7 \\ 2p^{-\frac14} (p-1), & p=3,5.\end{cases}
\end{equation*}
All together, 
$$|S_k^*(q,a)| \le  2^{\ell-1} 4 b^{-\frac14} \varphi(b) \le 8 q^{-\frac14}\varphi(q).$$

Now let $k>2$.  We consider $S_k^*(p^\ell, a)$.  If $\ell>\gamma$ then $S_k^*(p^\ell, a) = 0$.  So we may assume $\ell\le \gamma(p) \le k$.
If $p\le k$, then 
\begin{equation*}
|S_k^*(p^\ell,a)| \le  \varphi(p^{\ell})  = kp^{-1}\varphi(p^{\ell}) \le  k p^{-\ell/k} \varphi(p^{\ell}) .
\end{equation*}
If $p>k$, then $\gamma=1$ so $\ell =1$.  By Lemma 4.3 of \cite{Vaughan1997}, we have
\begin{align*}
|S_k^*(p,a)| & = |S_k^*(p,a) -1| \le (k-1)p^{1/2} + 1 \le kp^{1/2} \le k p^{-1/k}p^{5/6} \\ &
 = kp^{-1/k} \frac{p-1}{p^{1/6} - p^{-5/6}} \le \begin{cases} p^{-1/k}(p-1), & p>k^6 \\ kp^{-1/k}(p-1), & p\le k^6.\end{cases}
\end{align*}
\end{proof}

\bibliographystyle{amsplain}
%typesetting on laptop use:
%\bibliography{/Users/AylaGafni/Dropbox/library}
%typesetting on desktop use:
%\bibliography{/Users/Shared/Dropbox/library}

\end{document}